\documentclass[10pt]{article}
\font\smallit=cmti10

\usepackage[english]{babel}
\usepackage{amssymb,amsthm,amsmath,amsfonts,mathrsfs}
\usepackage{pst-all,pst-3dplot,pstricks,pstricks-add,pst-math,pst-xkey}
\usepackage{graphicx}
\usepackage{latexsym}
\usepackage{ifthen, bigstrut}
\usepackage{accents}
\usepackage{hyperref}
\usepackage{breakurl}

\usepackage{bbold}
\usepackage{enumerate}
\usepackage{epsfig}
\usepackage{subfigure}
\usepackage{skak}
\usepackage{chessboard}
\usepackage{chessfss}
\usepackage{pst-all}
\usepackage{lscape}
\usepackage{courier}
\usepackage[all,cmtip]{xy}

\usepackage{wasysym}

\newtheorem{theorem}{Theorem}

\newtheorem{corollary}[theorem]{Corollary}
\newtheorem{lemma}[theorem]{Lemma}

\let\leftmoon\relax

\usepackage{mathabx}
\newcommand{\conj}[1]{\oset{\curvearrowleftright}{#1}}
\newcommand{\conjS}[1]{\oset{\curvearrowleftright}{#1}}
	
\usepackage{bbold}
\newcommand{\bs}{\boldsymbol}
\newcommand{\N}{\ensuremath{\mathscr{N}}}
\renewcommand{\P}{\ensuremath{\mathscr{P}}
}\renewcommand{\L}{\mathscr{L}}

\newcommand{\oinf}{\ensuremath{\overline{\infty }}}
\renewcommand{\inf}{\ensuremath{{\bs\infty }}}

\newcommand{\Np}{{\mathbb{Np}}}
\newcommand{\Npi}{{\mathbb{Np}^\infty}}

\newcommand{\Npc}{\mathbb{Np}^{\rm C}}
\newcommand{\Imi}{\mathbb{Im}^\inf}
\renewcommand{\Im}{\mathbb{nim}}
\newcommand{\imm}[1]{S_{#1}}
\newcommand{\prot}[1]{P_{#1}}

\newcommand{\moonform}{\cg{\cg{\inf}{0}\!,0}{\cg{0}{\oinf}\!,0}}
\newcommand{\moonformsimple}{\cg{\inf}{\oinf}}
\newcommand{\mf}[4]{\cg{\cg{\inf}{#1},#2}{\cg{#3}{\oinf},#4}}

\newcommand{\moon}{\scalebox{1.1}{\ensuremath{\leftmoon}}}
\newcommand{\moonsmall}{\scalebox{0.9}{\ensuremath{\leftmoon}}}

\newcommand{\mex}{\ensuremath{\mathrm{mex}}}
\newcommand{\SG}{\mathcal{G}}

\newcommand{\GL}{{G^\mathcal{L}}}
\newcommand{\GR}{{G^\mathcal{R}}}

\newcommand{\XL}{X^\mathcal{L}}
\newcommand{\XR}{X^\mathcal{R}}

\newcommand{\cg}[2]{\left\{ #1\!\mid \!#2\right\}}

\theoremstyle{definition}
\newtheorem{definition}[theorem]{Definition}

\newtheorem{observation}[theorem]{Observation}

\newtheorem{example}[theorem]{Example}
\newtheorem{axiom}[theorem]{Axiom}
\newtheorem{notation}[theorem]{Notation}

\DeclareMathOperator{\cgfuzzy}{\|}

\makeatletter
\newcommand{\oset}[3][0ex]{%
  \mathrel{\mathop{#3}\limits^{
    \vbox to#1{\kern-1.5\ex@
    \hbox{$\scriptstyle#2$}\vss}}}}
\makeatother

\emergencystretch=\maxdimen
\begin{document}
\begin{center}
\uppercase{\bf Impartial games with entailing moves}
\vskip 20pt
{\bf Urban Larsson\footnote{urban031@gmail.com}}\\
{\smallit School of Computing, National University of Singpore, Singapore}\\
{\bf Richard J.~Nowakowski\footnote{r.nowakowski@dal.ca}}\\
{\smallit Department of Mathematics and Statistics, Dalhousie University, Canada}\\
{\bf Carlos P. Santos\footnote{Partially supported by UID/MAT/04721/2019 strategic project; cmfsantos@fc.ul.pt}}\\
{\smallit Center for Functional Analysis, Linear Structures and Applications,\\ University of Lisbon \& ISEL--IPL}\\
\end{center}

\begin{abstract}
Combinatorial Game Theory has also been called `additive game theory', whenever the analysis
involves sums of independent game components. Such {\em disjunctive sums} invoke comparison between games, which allows abstract values to be assigned to them.  However, there are rulesets with {\em entailing  moves} that break the alternating play axiom and/or restrict the other player's options within the disjunctive sum components. These situations are exemplified in the literature by a ruleset such as {\sc nimstring}, a normal play variation of the classical children's game {\sc dots\,\&\,boxes}, and {\sc top~entails}, an elegant ruleset introduced in the classical work Winning Ways, by Berlekamp Conway and Guy. Such rulesets  fall outside the scope of the established normal play theory. Here, we axiomatize normal play via two new terminating games, $\inf$ (Left wins) and $\oinf$ (Right wins), and a more general theory is achieved. We define {\em affine impartial}, which extends classical impartial games, and we analyze their algebra by extending the established Sprague-Grundy theory, with an accompanying minimum excluded rule. Solutions of {\sc nimstring} and {\sc top~entails} are given to illustrate the theory.
\end{abstract}

\section{Introduction}
Combinatorial Game Theory (CGT), as described in \cite{AlberNW2007,BerleCG1982,Con1976,Siegel2013}, considers disjunctive sums of normal play games. In order to evaluate the outcome of a sum of such games, it suffices to analyze the components individually, and then {\em add} the individual values. 

However, some classical impartial rulesets, such as {\sc nimstring} and {\sc top entails} fall slightly outside the usual CGT axioms. In {\sc nimstring}, certain moves require a player to play again, or {\em carry-on}, which is a violation of the alternating play axiom. And in {\sc top entails}, certain moves enforce the next player to play in the same component, which violates the standard definition of a disjunctive sum. Thus, values of individual components is no longer a relevant measure, given the standard CGT axioms. The type of moves mentioned in this paragraph will be gathered under the term  {\em entailing moves}.\footnote{Entailing means ``involve something as a necessary or inevitable part or consequence''.} 

The purpose of this paper is to extend impartial normal play games sufficiently to include games with entailing moves. While accomplishing this, we expand the classical Sprague-Grundy theory to fit this extension.

We will rebuild the normal play axioms by using so-called {\em terminating games}, or {\em infinities}, $\inf$ and $\oinf$. Here we focus on the impartial setting, and the general comprehensive theory for partizan games will appear in \cite{LNS}.\footnote{There are partizan rulesets, in the literature and in recreational play, with similar entailing and {\em terminating} moves. Probably the most prominent ones are the game of {\sc chess}, and a version of \textsc{go} called \textsc{atari go}.}  These theories are called \textit{affine impartial} and \textit{affine normal play} respectively. 

Although we consider only impartial games in this paper, we will keep the players distinguished as Left and Right. In particular, Left wins if either player plays to $\inf$, in any component, and Right wins in case of play to $\oinf$. Note that the normal play zero is restored by defining $0 = \cg{\oinf}{\inf}$, a first player losing position. 

It is well-known that, in classical Combinatorial Game Theory, the impartial values are nimbers. We will prove that there is exactly one more value modulo affine impartial, a game $\moon$, called the \emph{moon}. This value was anticipated in the classical work Winning Ways, by Berlekamp, Conway and Guy. In \cite{BerleCG1982}, volume 2, page 398, one can read ``A loony move is one that loses for a player, no matter what other components are.''. 
Before developing the theory, we illustrate how the infinities are used in the motivating rulesets, \textsc{nimstring} (\cite{Berle2000,BerleCG1982}) and \textsc{top~entails} (\cite{BerleCG1982}, volume 2).

Let us first briefly mention the organization of the paper. To facilitate the development of the new impartial theory, Section \ref{sec:formsorder} considers the basic properties of unrestricted affine normal play, aiming for a game comparison result, Theorem~\ref{th:compConway}.
The affine impartial theory is developed in Section \ref{sec:SG}. The main result is Theorem \ref{gsg} which shows that values in this extension are the nimbers plus one more value. Theorem~\ref{gmr} gives an algorithm to find the value of a given position, and notably, if there are no infinities in the options, then the nimbers are obtained by the usual mex-rule. We finish off with two case studies. In Section \ref{sec:nimstring},  we compute the value of an interesting \textsc{nimstring} position, anticipated in Section~\ref{sec:nimstringexample}. In Section \ref{sec:topentail}, we compute the values for {\sc top entail} heaps of sizes $1$ through $12$, and Theorem~\ref{th:slick} provides theoretical justification for computing {\sc top entails} values.

\subsection{The ruleset {\sc nimstring}}\label{sec:nimstringexample}
In \textsc{nimstring}, a player draws a line between two horizontally or vertically adjacent points, in a finite grid, not already joined by a line.
If a player completes a $1\times 1$ square, then they must draw another line, and if they cannot do so, they lose. 

Figure~\ref{fig:*2} shows an example position, where no square can be completed in the next move. Later, through the new theory, we will see that the position $H$ equals $*2$ modulo affine impartial. 
\begin{figure}[htbp]
\caption{A \textsc{nimstring} position, $H$.}
\label{fig:*2}
\begin{center}
\includegraphics{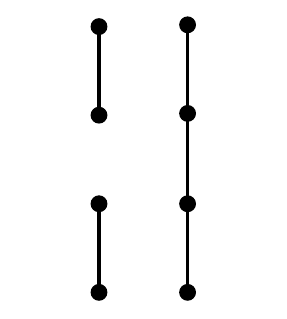}
\end{center}
\end{figure}

In Figure~\ref{fig:1}, we show a position, $G$, with two options, one of which is an entailing move. Namely, if the the top bar is drawn, the next player continues, but if the middle bar is drawn, then the current player has to carry-on.

\begin{figure}[ht]
\caption{A \textsc{nimstring} position, $G$, with its two options, a `double-box' and an entailing carry-on position.}
\label{fig:1}
\begin{center}
\includegraphics{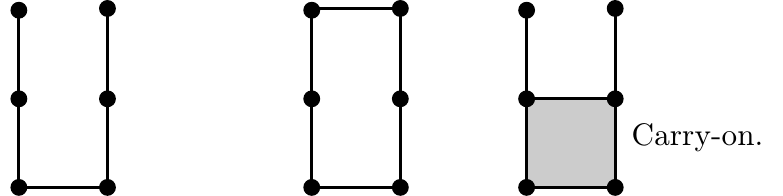}
\end{center}
\end{figure}

When we develop the theory, we will see that the position $G$, to the left in Figure~\ref{fig:1}, is the abstract game 
\begin{align}\label{earlymoonform}
\moonform. 
\end{align}
The option 0 is obtained by drawing the top bar. The intuition for this is as follows: the player who moves inside the `double box' loses, if $G$ is played in isolation, because they have to play again. 

If a player draws the middle bar in $G$, then they have to carry-on, and this is represented by the abstract option $\cg{\inf}{0}$, if Left moved. There is an infinite urgency in this game: Right has to play here, or lose.  And so, the effect is the desired: Left plays again, and  alternating play is restored. Hence disjunctive sum play is also restored, within the affine impartial convention. Moreover, the Right option in this threat should be 0, because Left loses by playing this option if $G$ is played alone. If the sum is $G+H$, with $H$ as in Figure~\ref{fig:*2}, then the next player wins, by playing this entailing middle bar in $G$.

\subsection{The ruleset {\sc top entails}}
\textsc{Top~entails} is played on heaps of tokens. A player may either remove the top token from exactly one heap, or split a heap into two non-empty heaps. If the top token is removed from a heap, then the next move (in alternating play) must be played on the same heap.

A heap with one token, say $H$, is a first player win, in any situation. Namely, a move in $H$ forces the opponent to play in the same heap, where no move remains. Note that the abstract game $H=\cg{\inf}{\oinf}$ settles this behaviour. The player who moves first in $H$ wins independently of existence of other components. The point we wish to make here is that this abstract representation settles the problem of independecny of a heap of size one with other disjunctive sum components.

Consider $G$, in Figure \ref{fig:ex7}, a pile of size 3. There are two options, as depicted in Figures~\ref{fig:ex8} and \ref{fig:ex9}. 
\begin{figure}[htbp]
\caption{A pile $G$ of \textsc{top entails}, of size 3.}
\label{fig:ex7}
\begin{center}
\includegraphics{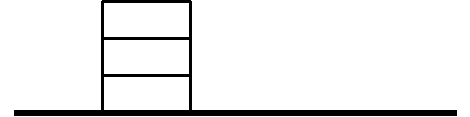}
\end{center}
\end{figure}

The option in Figure~\ref{fig:ex8} splits $G$ into two piles and the next player's options are unrestricted. By the terminating effect of playing in a heap of size one, this composite game should be equal to the game $H=\cg{\inf}{\oinf}$.

\begin{figure}[htbp]
\caption{The game $G$ is split into two components.}
\label{fig:ex8}
\begin{center}
\includegraphics{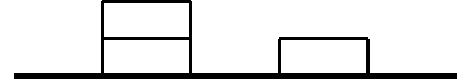}
\end{center}
\end{figure}

The option in Figure~\ref{fig:ex9} is an entailing move, and the next player must continue in this component, even if other moves are available. Therefore, the game form of the entailing option in Figure~\ref{fig:ex9} is $$\cg{\inf}{\bf 1+\bf 1,\bf 1_{\text{entail}}},$$ if Left just moved, and where {\bf 1} denotes a heap of size one. The terminating threat forces Right to play here, instead of possibly using other options.
\begin{figure}[htbp]
\caption{An entailing option.}
\label{fig:ex9}
\begin{center}
\includegraphics{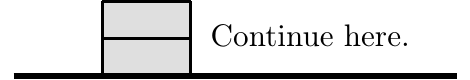}
\end{center}
\end{figure}

Intuitively, either way of responding reduces to a game of the form $H=\cg{\inf}{\oinf}$. In conclusion, the heap of size three should be equal to the game $H$, and disjunctive sum play has been restored. All this intuition will be rigorously justified  in the coming complete theory for affine impartial play.

It turns out that affine impartial games require only a small extension to the Sprague-Grundy theory. Namely, the game in \eqref{earlymoonform}, obtained from the {\sc nimstring} position in Figure~\ref{fig:1}, equals the game $H=\cg{\oinf}{\inf}$ in the previous paragraph, modulo affine impartial, and later we will devote the value `$\moon$' to the equivalence class of such games.

\section{Affine literal forms and order}\label{sec:formsorder}

This section aims at Theorem~\ref{th:invcomparison}, a comparison theorem for affine normal play that suffices for the purpose of this paper. We begin by defining the fundamental concepts for affine normal play, and we wait with the restriction to affine impartial until the next section.

In classical Combinatorial Game Theory, the normal play forms, $\mathbb{Np}$, are recursively constructed from the empty set. The form $\{\varnothing\,|\,\varnothing\}=0$ is the only form of day zero and the only form without options. The forms $\{0\,|\,\varnothing\}=1$ , $\{\varnothing\,|\,0\}=-1$, $\{0\,|\,0\}=*$, are born on day 1, and so on.

The forms of {\em affine normal play}, denoted $\mathbb{Np}^\infty$, are recursively constructed from the games $\infty$ (infinity) and $\overline{\infty}$ (minus infinity) \cite{LNS}. The forms $\infty$ and $\overline{\infty}$ are the only forms without options. The forms $\cg{\oinf}{\inf}=0,\cg{\inf}{\oinf}=\pm\inf, \cg{\inf}{\inf}$ and $\cg{\oinf}{\oinf}$
are born on day zero. And so on.

The order of $\mathbb{Np}^\infty$ is defined in the standard way. Consider the four perfect play outcome classes $\mathscr{L}$ (Left wins), $\mathscr{N}$ (Next player wins), $\mathscr{P}$ (Previous player wins), and $\mathscr{R}$ (Right wins). From Left's perspective, the first outcome is the best (she wins, regardless of playing first or second) and the fourth is the worst (she loses, regardless of whether playing first or second). On the other hand, regarding $\mathscr {N}$ and $\mathscr{P}$, the victory depends on playing first or second, so these outcomes are not comparable. These considerations explain the partial order in an `outcome diamond':

\begin{center}
\includegraphics[scale=1.2]{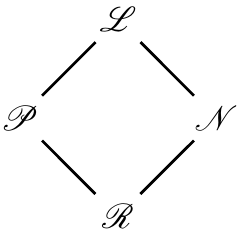}\\
\end{center}

We write $G\in \mathscr{L}$, or equivalently $o(G)=\mathscr L$, if the outcome of $G\in \mathbb{Np}^\infty$ is Left wins, and so on. The evaluation of games in $\mathbb{Np}^\infty$ is based on the following axiomatic list:

\begin{axiom}[Absorbing Nature of Infinities]\label{axiom}
The infinities satisfy
\begin{enumerate}
  \item $\infty\in\mathscr{L}$;
    \item $\overline{\infty}\in\mathscr{R}$;
          \item For all $X\in \mathbb{Np}^\infty\setminus \{\oinf\}$, $\infty+X=\infty$;
           \item For all $X\in \mathbb{Np}^\infty\setminus \{\inf\}$, $\oinf+X=\oinf$;
  \item `$\infty+\overline{\infty}$' is not defined.
\end{enumerate}
\end{axiom}
Addition of games is defined as usual, apart from items 3 and 4.
The fifth item is natural in terms of perfect play, since if \inf\ appears, then \oinf\ cannot appear and vice versa.

The definitions of equality and partial order of games are based on the outcome diamond.
%

\begin{definition}[Order and Equality of Games]\label{order}
Let $G, H\in\mathbb{Np}^\infty$. Then, $G\succcurlyeq H$ if, for every form $X\in \mathbb{Np}^\infty\setminus\{\infty,\overline{\infty}\}$, $o(G+X)\geqslant o(H+X)$. Moreover $G=H$ if $G\succcurlyeq H$ and $H\succcurlyeq G$.
\end{definition}
Note that the exclusion of the infinities does not diminish the generality of the definition, but is necessary due to Axiom~5.
As usual, we have the following observations.
 If $G=H$ then replacing $H$ by $G$ or $G$ by $H$ do not hurt the players under any circumstances.
Similarly, if $G\succcurlyeq H$ then replacing $H$ by $G$ does not hurt Left, and replacing $G$ by $H$ does not hurt Right.

\begin{theorem}\label{inflarge}
Let $G\in\mathbb{Np}^\infty$. Then $\infty\succcurlyeq G$ and $G\succcurlyeq\overline{\infty}$.
\end{theorem}

\begin{proof}
If $X\in \mathbb{Np}^\infty\setminus\{\infty,\overline{\infty}\}$ then, by Axiom~3, $\infty+X=\infty$. Hence, by Axiom~1, $o(\infty+X)=o(\infty)=\mathscr{L}$. Therefore, for every $X\in \mathbb{Np}^\infty\setminus\{\infty,\overline{\infty}\}$,  we have $o(\infty+X)\geqslant o(G+X)$, and so $\infty\succcurlyeq G$. Proving that $G\succcurlyeq\overline{\infty}$ is analogous.
\end{proof}

The concept of a {\em check} is  fundamental to $\mathbb{Np}^\infty$. Indeed, this is an alternative, and perhaps more explicit, at least for those Chess playing readers, term for an {\em entailing move}, as seen in the Introduction.

\begin{definition}[Check Games] Consider $G\in \mathbb{Np}^\infty$. If $\infty\in \GL$ ($\overline{\infty}\in \GR$) then $G$ is a \emph{Left-check} (\emph{Right-check}). If $G$ is a Left-check or a Right-check then $G$ is a \emph{check}. Denote by $G^{\oset{\rightarrow}{L}}$ ($G^{\oset{\leftarrow}{R}}$) a Left (Right) option of $G$ that is a Left-check (Right-check).
\end{definition}

Of course, all checks are asymmetric, apart from the `trivial check', $\cg{\inf}{\oinf}$. A player would not use this check, because the opponent `check mates' by defending.

\begin{definition}[Quiet Games]\label{def:quiet}
Let $G\in\mathbb{Np}^\infty$. If $G\neq \inf$ ($G\neq \oinf$) and $G$ is not a Left-check (Right-check) then $G$ is \emph{Left-quiet} (\emph{Right-quiet}). If $G$ is Left-quiet and Right-quiet then $G$ is \emph{quiet}.
\end{definition}


\begin{definition}[Conway Forms and Games]\label{def:conw}
A game $G\in \mathbb{Np}^\infty$ is a \emph{Conway form} if $G\not\in\{\infty,\overline{\infty}\}$, and $G$ has no checks as followers. Let $\mathbb{Np}^{\rm C}\subseteq\mathbb{Np}^\infty$ denote the substructure of Conway forms. A game $G\in \mathbb{Np}^\infty$ is a \emph{Conway game} if it equals a Conway form.
\end{definition}

\begin{example}
The game $G=\cg{\cg{\oinf}{\infty}}{\cg{\oinf}{\infty}} =\cg{0}{0} =*$ is a Conway form (no checks as followers). 
The game $G'=\cg{\cg{\infty}{*}}{\cg{*}{\oinf}}$ is not a Conway form because there are checks as followers. However, later, we will see that $G'=G$. Therefore, $G'$ is a Conway game.
\end{example}

In general, when we say form, we mean the literal form, and when we say game, we usually mean (any member in) the full equivalence class of games. When we write $G\in \mathbb{Np}^\infty$, we usually refer to the literal form, but the context may decide.

Some classical theorems are still available in $\mathbb{Np}^\infty$.

\begin{theorem}[Fundamental Theorem of Affine Normal Play]\label{thm:ftnp}
If $G\in \mathbb{Np}^\infty$ then $G\succcurlyeq 0$ if and only if $G\in \mathscr{L}\cup\mathscr{P}$.
\end{theorem}

\begin{proof} Assume that $G\succcurlyeq 0$. We have $0 \in \mathscr{P}$, and so, by order of outcomes, $G\in \mathscr{L}\cup\mathscr{P}$.

Suppose now that $G\in \mathscr{L}\cup\mathscr{P}$. If $G=\infty$, by Theorem \ref{inflarge}, $G\succcurlyeq 0$; hence, assume $G\neq\infty$. Let $X\in \mathbb{Np}^\infty\setminus\{\infty,\overline{\infty}\}$.

If, playing first, Left wins $X$ with the option $X^L$, then she also wins $G+X$ with the option $G+X^L$. Essentially, she mimics the strategy used when $X$ is played alone, answering locally when Right plays in $G$. Due to the assumption $G\in \mathscr{L}\cup\mathscr{P}$, this is a winning strategy for Left in $G+X$.

If Left, playing second, wins $X$. Then, on $G+X$, she can respond to each of Right's moves locally, with a winning move on the same component, because $G\in \mathscr{L}\cup\mathscr{P}$. Thus Left can win $G+X$ playing second.

Therefore,  $o(G+X)\geqslant o(X)$ and so, $G\succcurlyeq 0$.
\end{proof}

\begin{corollary}[Order-Outcome  Bijection]\label{thm:ftnp2}
If $G\in \mathbb{Np}^\infty$ then
\begin{itemize}
  \item $G\succ0$ if and only if $G\in \mathscr{L}$;
  \item $G=0$ if and only if $G\in \mathscr{P}$;
  \item $G\cgfuzzy 0$ if and only if $G\in \mathscr{N}$;
  \item $G\prec0$ if and only if $G\in \mathscr{R}$.
\end{itemize}
\end{corollary}

\begin{proof} The statement of Theorem~\ref{thm:ftnp} can equivalently be ``$G\preccurlyeq 0$ if and only if $G\in\mathscr{R}\cup\mathscr{P}$'', so we can use that fact too.

Suppose that $G\succ0$. By Theorem \ref{thm:ftnp}, $ G\in\mathscr{L}\cup\mathscr{P}$. But, we cannot have $G\in\mathscr{P}$, for otherwise $G\in\mathscr{R}\cup\mathscr{P}$ and $G\preccurlyeq 0$. Therefore, $G\in\mathscr{L}$. Conversely, suppose that $G\in\mathscr{L}$. By Theorem \ref{thm:ftnp}, we have $G\succcurlyeq 0$. But, we cannot have $G=0$, for otherwise $G\preccurlyeq 0$, and  $G\in\mathscr{R}\cup\mathscr{P}$. Hence, $G\succ 0$. Thus, the first equivalence holds.

The proof of the fourth equivalence is analogous.

For the second equivalence, if $G=0$, then $G\succcurlyeq 0 \;\wedge\; G\preccurlyeq 0$. So, $G\in(\mathscr{L}\cup\mathscr{P})\cap (\mathscr{R}\cup\mathscr{P})=\mathscr{P}$.

The third equivalence is a consequence of eliminating all other possibilities.
\end{proof}

It is known that $\mathbb{Np}$ is a group. By Corollary \ref{thm:ftnp2} we may deduce that $\mathbb{Np}^\infty$ is only a monoid. Namely, if $G=\{\infty\,|\,0\}$ then, for any $X\in \mathbb{Np}^\infty\setminus\{\infty,\overline{\infty}\}$, $G+X\in \mathscr{L}\cup\mathscr{N}$ (playing first, Left wins). Hence, for all $X$, $G+X\neq 0$ and $G$ is non-invertible. Thus, in general, the comparison of $G$ with $H$ cannot be done by playing the game `$G-H$', because, sometimes, `$-H$' does not exist.

However, the following theorem shows that not everything is lost.

The \emph{conjugate} of a given game switches roles of the players.
\begin{definition}[Conjugate]\label{def:conjugate}
The conjugate of $G\in \Npi$ is
\[ \conj{G}\, =
\begin{cases}
\oinf, &\mbox{if $G=\infty$}\\
\infty, &\mbox{if $G=\oinf$}\\
\cg{\conjS\GR\, }{\; \conjS\GL }, &\mbox{otherwise},
\end{cases}
\]
where $\conjS{\GL}$ denotes the set of literal forms
$\conjS{G^L} $, for $G^L\in \GL$, and similarly for $\GR$.

\end{definition}

\begin{theorem}\label{th:conjConway}
If $G\in\mathbb{Np}^\infty$ is a Conway game, then $G$ is invertible and $-G=\,\conj{G}$.
\end{theorem}

\begin{proof}
Suppose first that $G$ is a Conway form. If $G=0$ then $\conj{G}=0$, and the\linebreak theorem holds. Otherwise, let us verify that $G+\conj{G}$ is a $\mathscr{P}$-position. If Left, playing first, chooses $G^L+\conj{G}$, because this game is not $\infty$ ($G$ is not a check), Right can answer with $G^L+\conj{G^L}$ and, by induction, because $G^L$ is a Conway form with no checks as followers, that option is equal to zero. Because of that, by Corollary~\ref{thm:ftnp2}, that option is a $\mathscr{P}$-position, and Right wins. Analogous arguments work for the other options of the first player, and so, $G+\conj{G}$ is a $\mathscr{P}$-position. Again, by Corollary~\ref{thm:ftnp2}, $G+\conj{G}=0$.

Suppose now that $G$ is not a Conway form. Because it is a Conway game, by definition, it is equal to some $G'\in\Npc$. The first paragraph proved that $G'+\conj{G'}=0$. Also, by symmetry, $\conj{G}$ is equal to $\conj{G'}$. Therefore, $G'+\conj{G'}=0$ implies $G+\conj{G}=0$.
\end{proof}

\begin{lemma}\label{th:compinvertible}
Let $G,H\in\mathbb{Np}^\infty$, and let $J$ be an invertible form of $\mathbb{Np}^\infty$. Then $$G\succcurlyeq H\Leftrightarrow G+J\succcurlyeq H+J.$$
\end{lemma}

\begin{proof} ($\Rightarrow$) Consider any $X\in \mathbb{Np}^\infty\setminus\{\infty,\overline{\infty}\}$ and let $X'=J+X $. Since $J$ is invertible, $J$ is neither $\infty$ nor $\overline{\infty}$, and so, $X'$ is neither $\infty$ nor $\overline{\infty}$.  Definition of order implies $o(G+X')\geqslant o(H+X')$, that is $o(G+J+X))\geqslant o(H+J+X)$. Thus, the arbitrariness of $X\in \mathbb{Np}^\infty\setminus\{\infty,\overline{\infty}\}$ implies $G+J\succcurlyeq H+J$.\\

\noindent
($\Leftarrow$) Consider any $X\in \mathbb{Np}^\infty\setminus\{\infty,\overline{\infty}\}$ and let $X'=-J+X$ ($J$ is invertible, i.e. $-J $ exists and $J-J=0$). Since $-J$ is invertible, $-J$ is neither $\infty$ nor $\overline{\infty}$, and so, $X'$ is neither $\infty$ nor $\overline{\infty}$. By definition of order, $o(G +J+X')\geqslant o((H+J+X')$, that is $o(G+J-J+X)\geqslant o(H+J-J+X)$. Hence, $o(G+X)\geqslant o(H+X)$, and so, given the arbitrariness of $X\in \mathbb{Np}^\infty\setminus\{\infty,\overline{\infty}\}$, $G \succcurlyeq H$. \end{proof}

\begin{theorem}\label{th:invcomparison}
Let $G$ be any form of $\mathbb{Np}^\infty$ and suppose that $H$ is an invertible form of $\mathbb{Np}^\infty$. Then, $$G\succcurlyeq H\Leftrightarrow G-H\in\mathscr{L}\cup\mathscr{P} \text{ and }G=H\Leftrightarrow G-H\in\mathscr{P}.$$
\end{theorem}

\begin{proof} By Lemma \ref{th:compinvertible}, $G\succcurlyeq H \Leftrightarrow G-H \succcurlyeq H-H$. Therefore, we have $G\succcurlyeq H\Leftrightarrow G-H \succcurlyeq 0$. By Theorem \ref{thm:ftnp}, this is the same as $G\succcurlyeq H \Leftrightarrow G-H \in \mathscr{L} \cup \mathscr{P}$.

Finally, $G=H\Leftrightarrow G-H\in\mathscr{P}$, by $G\succcurlyeq H\wedge H\succcurlyeq G$.
\end{proof}

\begin{theorem}\label{th:compConway}
Let $G$ be any form of $\mathbb{Np}^\infty$ and let $H\in\mathbb{Np}^\infty$ be a Conway game. Then
\begin{itemize}
\item $G\succcurlyeq H$ if and only if $G\,+\conj{H}\,\in\mathscr{L}\cup\mathscr{P}$
\item $G=H$ if and only if $G\,+\conj{H}\,\in\mathscr{P}$.
\end{itemize}
\end{theorem}

\begin{proof} These are direct consequences of Theorems \ref{th:conjConway} and \ref{th:invcomparison}.
\end{proof}

In a follow up paper \cite{LNS}, where we study the full game space $\mathbb{Np}^\infty$, we provide a solution of the general case of $G\succcurlyeq H$.
\section{Affine impartial theory}\label{sec:SG}

In order to propose an extension of the Sprague-Grundy theory, we first define the concept of an affine impartial game.\footnote{In terms of ruleset: here `affine impartial' is an abbreviation of affine normal play impartial, in the sense that if the player to move cannot complete their move they lose.} Of course, rulesets like {\sc nimstring} should be impartial.

\begin{definition}[Symmetric Game]\label{def:symmetric}
Consider a form $G\in\Npi$. Then $G$ is \emph{symmetric} if $G\not\in\{\infty,\overline{\infty}\}$ and $\GR=\; \conjS{\GL}$.
\end{definition}

\begin{definition}[Affine Impartial]\label{def:impartial}
A form $G\in\mathbb{Np}^\infty$ is \emph{affine impartial} if it is symmetric and all quiet followers of $G$ are symmetric. The subset of affine impartial games is  $\Imi\subset\Npi$.
\end{definition}

Of course, a non-quiet game either has no option, or is a check, and so (unless a trivial check) is by definition asymmetric. But this is the only exception of symmetry in the world of affine impartial impartial games. It is easy to check that $\Imi$ satisfies the standard closure properties of combinatorial games, i.e. closure of taking options, addition, and conjugates.

The following result must hold for any class of games that claims to be ``impartial''.

\begin{theorem}[Affine Impartial Outcomes]\label{symmetric}
If $G$ is a symmetric form, then $G\in\mathscr{N}\cup\mathscr{P}$.
\end{theorem}

\begin{proof} This proof uses a strategy-stealing argument. Suppose that $G\in\mathscr{L}$. Then Left wins $G$ playing first with some option $G^L$. Hence, by symmetry, Right wins $G$ playing first with $\conjS{G^L}$. That contradicts $G\in\mathscr{L}$. A similar argument holds against $G\in \mathscr{R}$.
\end{proof}

We want to restrict our analysis to $\Imi$. Therefore, we define equality modulo $\Imi$.

\begin{definition}[Impartial Equality]\label{equalityI}
Consider forms $G, H\in \Imi$. Then, $G=_{\Imi}H$ if, for every form $X\in\Imi$, $o(G+X)=o(H+X)$.
\end{definition}


\begin{observation}
Of course, $G=H$ in $\mathbb{Np}^\infty$ implies $G=_{\Imi}H$. The opposite direction is not true. We can have $G=_{\Imi}H$ and $G\neq H$ in $\mathbb{Np}^\infty$, if the there is no distinguishing game in $\Imi$. A simple example is $G =\moonform $ and $H =\mf{*}{*}{*}{*} $. As we will see, these games are indistinguishable modulo $\Imi$. However, the game $X=\cg{0}{-1}$ distinguishes them in $\Npi$; playing first, Left wins $G+X$,  but loses $H+X$.
\end{observation}

It is easy to verify if a form in $\Imi$ equals a nimber.

\begin{theorem}[Nimbers]\label{nimbers}
Let $G\in\Imi$. Then, $G=_{\Imi}*n$ if and only if $G+*n\in\P$.
\end{theorem}

\begin{proof}
Suppose that $G+*n\in\mathscr{P}$. By Theorem~\ref{th:compConway}, $G=*n$ modulo $\mathbb{Np}^\infty$, and so $G=_{\Imi}*n$.

Suppose now that $G=_{\Imi}*n$. By Theorem \ref{symmetric}, $G+*n\in\mathscr{N}\cup\mathscr{P}$, since impartiality is closed under addition. If $G+*n\in\mathscr{N}$, then $G+*n\in\mathscr{N}$ and $*n+*n\in\mathscr{P}$, contradicting $G=_{\Imi}*n$. Hence, $G+*n\in\mathscr{P}$.
\end{proof}

\begin{notation}
Let $\Im\subseteq\Imi$ denote the subset of affine impartial games that equal nimbers.
\end{notation}

It is well-known that, in classical Combinatorial Game Theory, the impartial values are nimbers. We will prove that there is exactly one more value modulo $\Imi$, a game $\moon$, called \emph{moon}. In \cite{BerleCG1982}, volume 2, page 398, one can read ``A loony move is one that loses for a player, no matter what other components are.''. The following general definition is motivated by that idea.

\begin{definition}[Loony Game]
A game $G\in \Npi$ is \emph{loony} if, for all quiet $X\in \Npi\cap(\N\cup \P)$, $G+X\in \mathscr{N}$.
\end{definition}

 Thus, in our interpretation, a `loony move' exposes a loony game.

There are no loony games in $\mathbb{Np}$. Suppose that $G\in \Np\cap(\mathscr{P}\cup\mathscr{L}\cup\mathscr{R})$ is a loony game. Of course, $G+0\in\mathscr{P}\cup\mathscr{L}\cup\mathscr{R}$ and that is a contradiction. Suppose that $G\in \Np\cap\N$ is a loony game. In that case, if $n$ is large enough, $G+\cg{n}{0}\in\L$, and that is a contradiction, since $\cg{n}{0}\in\N$ is quiet. 

There are loony games in $\mathbb{Np}^\infty$ . The obvious one is $\pm\infty=\cg{\inf}{\oinf}$, but we can also have impartial quiet loony moves. Consider $G=\moonform$ and a quiet $X\in \mathbb{Np}^\infty$ such that $X\in\P\cup \N$. If $X\in \P$, the first player wins moving to $X$. If $X\in \N$, the first player wins moving to $\{\infty\,|\,0\}+X$ (Left) or to $\{0\,|\,\overline{\infty}\}+X$ (Right).

\begin{notation}
The {\em moon} is the game form $\moon =\moonformsimple$. 
\end{notation}

When a player moves to $\moon+X$, for any $X\in\mathscr{N}\cup \mathscr{P}$, he ``goes to the moon'' and loses.

\begin{theorem}[Loony Uniqueness]\label{loo}
All loony games are equal modulo $\Imi$.
\end{theorem}
\begin{proof} Consider $G$ and $G'$, two loony games. We know that all quiet $X\in\Imi$ belong to $\mathscr{N}\cup \mathscr{P}$. By  definition of a loony game, we have $G+X\in \mathscr{N}$ and $G'+X\in \mathscr{N}$. On the other hand, if $X\in\Imi$ is not quiet then $\infty\in \XL$ and $\overline{\infty}\in \XR$, and hence $G+X\in \N$ and $G'+X\in \N$. In all cases, $o(G+X)=o(G'+X)=\N$ and the theorem holds.
\end{proof}
\begin{observation} Two loony games may be different modulo $\Npi$, but equal modulo $\Imi$. The games $\moonform$ and $\mf{2}{0}{-2}{0}$  are loony. These games are different modulo $\mathbb{Np}^\infty$. Left, playing first, loses $\moonform-1$ and wins $\mf{2}{0}{-2}{0}-1$. However, as will follow by theory developed here, one cannot distinguish these two games modulo $\Imi$.
\end{observation}

In order to prove an affine impartial minimum excluded rule, we separate the options into two classes. 

\begin{definition}[Immediate Nimbers]\label{def:imm}
Let $G\in\Imi$. The set of $G$-\emph{immediate nimbers}, denoted $\imm{G}$ is the set $S_G=\GL\cap\Im$.
\end{definition}
Not that, by symmetry, $S_G=\GR\cap \Im$, and note that $S_{\moonsmall}=\varnothing$.
\begin{definition}[Protected Nimbers]\label{def:prot}
Conisder a game form $G\in\Imi$. The set of $G$-\emph{protected nimbers} $\prot{G}$ is
\begin{enumerate}
  \item $\prot{G}=\Im$, if $\infty\in \GL$;
  \item $\prot{G}=\{*n : G^{\oset{\rightarrow}{L}}+*n\in\L, G^{\oset{\rightarrow}{L}}\in\GL\}$, otherwise.
\end{enumerate}
\end{definition}
The second item says: if $\infty\not\in \GL$ then $*n\in P_G$ if there is a check $G^{\oset{\rightarrow}{L}}=\{\infty\,|\,G^{L\mathcal{R}}\}\in \GL$ such that Right, playing first, loses $G^{\oset{\rightarrow}{L}}+*n$. That is, playing first, Left is protected against those nimbers in a disjunctive sum.

Similar to Definition~\ref{def:imm}, we could have defined $\prot{G}$ with respect to Right options, to obtain the same set. 

Note that $P_{\moonsmall}=\Im$. This statement holds for the literal form $\moon = \pm\infty$. However, one can show that by using instead the form $\moon=\moonform$, as in \eqref{earlymoonform}, then $P_{\moonsmall}=\Im\setminus \{0\}$. The output of ``protected'' is sensitive to which form we choose. 

When the underlying game form is understood, we simply refer to the immediate and protected nimbers, respectively.

\begin{example}
Let $G\in\Imi$ be such that the Left options are $0$, $*2$, and $\{\infty\,|\,\{*\,|\,\overline{\infty}\},0\}$. Of course, $S_G=\{0,*2\}$. On the other hand, playing first, Left can use the check to win $G+*$. Because of that, $P_G=\{*\}$. An important observation is that, although Left is protected against the nimber $*$, Left cannot force a Left move to $*$ in $G$. But if Right moves to 0, Left wins $G+*$ anyway
\end{example}
%

Sometimes, Right can manoeuvre Left's eventual play to a nimber, or worse, via a sequence of `check upon check'.

\begin{definition}[Manoeuvrable Form]\label{def:maneu}
A quiet form $G\in\Imi$ is \emph{manoeuvrable} if after each Left move that is neither a nimber nor $\oinf$, Right can force, with checks, a Left move to a nimber or a move by either player to $\oinf$.
\end{definition}

\begin{example}
The form $G=\{*2,\{\infty\,|\,\{0,*4\,|\,\overline{\infty}\},0\}\,|\,*2,\{0,\{\infty\,|\,0,*4\}|\,\overline{\infty}\}\}$ is manoeuvrable. If Left avoids the immediate nimber $*2$, by checking, then Right can still force Left to move to one of the nimbers $0$ or $*4$.
\end{example}

\begin{lemma}\label{ff}
If $G\in\Imi$ is manoeuvrable, then $\prot{G}$ is finite.
\end{lemma}

\begin{proof}
After a Left first move in $G$, if needed, Right can force with checks a Left move to a nimber or a move by either player to $\oinf$.  Let $C$ be the set of nimbers that can arise through this forcing strategy by Right. Then $C$ is finite, because we study short games. Let $*n$ be a nimber such that, for all $*m\in C$, we have $n>m$. In $G+*n$, after a first check, say, to $G^L+*n$, Right forces with checks a move by either player to $\overline{\infty}$ or a Left move to $*m+*n$ ($n>m$). In the second case, after the sequence, Right wins with a TweedleDee-TweedleDum move. Thus, Left can protect against at most a finite number of nimbers. That explains why $P_G$ is finite in case of manoeuvrable games.
\end{proof}

Let $\mathrm{mex}(X)$ denote the smallest nonnegative integer not in $X$. Let $\mathcal{G}$ denote the set of Sprague-Grundy values of a set of nimbers, i.e. if $S=\{*n_i\}$, then $\mathcal{G}(S)=\{n_i\}$.

\begin{lemma}\label{man}
If $G\in\Imi$ is manoeuvrable then $G$ equals the nimber $*n$, where $n=\mathrm{mex}(\mathcal{G}(S_G\cup P_G))$.
\end{lemma}

\begin{proof}
By Lemma \ref{ff}, we know that $S_G\cup P_G$ is finite. Let $n=\mathrm{mex}(\mathcal{G}(S_G\cup P_G))$. Let us argue that the game $G+*n\in\mathscr{P}$. If the first player moves in $G$ to a nimber $*m\in S_G$, because $n$ is excluded from $\mathcal{G}(S_G)$, he loses. If the first player moves in $G$ to a quiet not nimber $G'$, because $G'$ is not a nimber, $G'+*n\in\mathscr{N}$ (Theorem~\ref{nimbers}), and the first player also loses. If the first player moves in $G$, giving a check, because  $n$ is excluded from $\mathcal{G}(P_G)$, he also loses. Finally, if the first player moves to $G+*n'$ ($n'<n$), because $n$ is the minimum excluded from $\mathcal{G}(S_G\cup P_G)$, he loses because the opponent has a direct TweedleDee-TweedleDum move or wins with a check. Hence, by Theorem~\ref{nimbers}, $G=*n$.
\end{proof}

\begin{lemma}\label{tl}
If $G,H\in\Imi$ are not nimbers, then $G+H\in\N$.
\end{lemma}

\begin{proof}
Consider $G,H\in\Imi\setminus \Im$. For a contradiction, assume that the sum of the birthdays, $b=b(G)+b(H)$, is the smallest possible such that $G+H\in\P$. Note that, by the assumptions on $G$ and $H$, $b>0$. Without loss of generality, we will analyze the move from $G+H$ to $G+H^L$. First, we prove two claims that concern local play in $G$ and $H$ respectively.\\

\noindent {\bf Claim 1.} Playing second in $G$, Left can avoid Left moves to nimbers and moves by either player to $\overline{\infty}$ until the first Right-quiet move.\\

\noindent {\em Proof of Claim 1.}
Suppose that Right, playing first in $G$, could force a Left move to a nimber or a move by either player to $\oinf$. If so, in $G+H$, by giving checks in $G$, Right could force some $G^{\oset{\leftarrow}{R}L\cdots\oset{\leftarrow}{R}L}+ H= *n+H$ (Right's turn) or a move by either player to $\oinf + H$. Of course, the second situation would be a victory for Right. Regarding the first case, at that moment, the position would be $*n+H$. And, because $H$ is not a nimber, by Theorem~\ref{nimbers}, we would have $*n + H \in \N$, which is a winning move for Right. In either case, Right, as first player, would win. That would contradict $G + H \in \P$.\\


\noindent {\bf Claim 2.} There is an $H^L$ such that Left can avoid Left moves to nimbers and moves by either player to $\overline{\infty}$, until the first Right-quiet move.\\

\noindent {\em Proof of Claim 2.} This is exactly the same as saying that $H$ is non-manoeuvrable. If it was  manoeuvrable, by Lemma \ref{man}, it would be a nimber, and we would have a contradiction again.\\
%
%

Let us return to the move from $G+H$ to $G+H^L$. Because $G+H\in\P$, Right has a winning move from $G+H^L$. But, by Claims 1 and 2, Left can play such that, at any stage before a Right-quiet move, Right is moving on $g+h$, where $g$ is a follower of $G$ and $h$ is a follower of $H^L$, such that neither $g$ nor $h$ is a nimber.

Either way, by assumption, there is a winning quiet Right-move $g^R+h$ or $g+h^R$. Since these are impartial games, we must have $g^R+h\in\P$ or $g+h^R\in\P$. But, because $h$ and $g$ are not nimbers, it follows by Theorem~\ref{nimbers} that $g^R$ and $h^R$ are not nimbers.

Therefore, we have $g^R+h\in\P$ or $g+h^R\in\P$ with both components not nimbers. But this contradicts the smallest birthday assumption. The result follows.
\end{proof}

\begin{theorem}[Affine Impartial Values]\label{gsg}
Every affine impartial form equals a nimber or the game $\moon\pmod{\Imi}$.
\end{theorem}
\begin{proof} Let $G\in\Imi$. If there is some $*n$ such that $G+*n\in\P$, then $G=_{\Imi}*n$, by Theorem \ref{nimbers}.

Suppose next that $G+*n\in\N$, for all $n$, so that $G$ does not equal a nimber modulo $\Imi$. By Lemma~\ref{tl}, for all $X\in\Imi\setminus \Im$, we also have $G+X\in \N$. Hence, for all $X\in\Imi$, we have $G+X\in\N$, and therefore $G$ is a loony game. Because, by Theorem \ref{loo}, all loony games are equal modulo $\Imi$, and $\moon$ is a loony game, we have $G=_{\Imi}\moon$.
\end{proof}

\begin{observation}
A form can be loony modulo $\Imi$ and not loony modulo $\Npi$. An example is the form $G=\cg{*,\cg{\inf}{*}}{*,\cg{*}{\oinf}}$. This game is not loony modulo $\Npi$ because, if $X=\cg{0}{-1}\in\mathscr{N}$, playing first, Left loses $G+X$. However, $G=_{\Imi} \moon$. This follows, by Theorem~\ref{gsg}, since $G$ does not equal any nimber;  if Right starts $G+*n$, he wins, by an appropriate parity consideration.
\end{observation}

\begin{theorem}\label{absorbing}
The game $\moon$ is absorbing modulo $\Imi$, that is, $\moon+Y=_{\Imi}\moon$, for all $Y\in \Imi$.
\end{theorem}

\begin{proof}
Since $\moon=\pm\infty$, regardless of what $X\in \Imi$ is, the first player wins both $\moon+Y+X$ and $\moon+X$. Therefore, by definition of equality of games, $\moon+Y=_{\Imi}\moon$.
\end{proof}

\begin{corollary}
The game $\moon$ is an idempotent modulo $\Imi$, that is, $\moon+\moon=_{\Imi}\moon$.
\end{corollary}

\begin{proof}
This is a trivial consequence of Theorem~\ref{absorbing}.
\end{proof}

%

\begin{definition}
The Sprague-Grundy value of the moon is $\mathcal{G}($\moon$)=\infty$.
\end{definition}

The following theorem explains how the Sparague-Grundy value of $G\in\Imi$ is determined by the set $S_G\cup P_G$.

\begin{theorem} [Affine Impartial Minimum Excluded Rule]\label{gmr}
Let $G\in\Imi$. We have the following possibilities:
\begin{itemize}
  \item If $S_G\cup P_G=\Im$, then $G=\moon$ and $\mex(\mathcal{G}(S_G\cup P_G))=\infty$;
  \item If $S_G\cup P_G\neq\Im$, then $G=*\left(\mex(\mathcal{G}(S_G\cup P_G))\right)$.
\end{itemize}
\end{theorem}

\begin{proof} If $\imm{G}\cup \prot{G}=\Im$, we have $G+*n\in\N$ for all $n$. Because of that, $G$ is not a nimber and, by Theorem~\ref{gsg}, $G=\moon$. If $\imm{G}\cup \prot{G}\neq\Im$, we use the same argument of the proof of Lemma ~\ref{man}.
\end{proof}

\begin{corollary}
If all the options of a game $G\in\Imi$ are quiet then $G$ is a nimber.
\end{corollary}

\begin{proof}
\noindent
If all the options of a game $G\in\Imi$ are quiet, then $\prot{G}=\varnothing$. Therefore, $\imm{G}\cup \prot{G}=\imm{G}\neq\Im$ and, by Theorem~\ref{gmr}, $G=*\left(\mex(\SG(\imm{G}))\right)$.
\end{proof}

\section{Case study: {\sc nimstring}}\label{sec:nimstring}

In the introduction, we promised to show that the following component equals $*2$.

\begin{center}
\includegraphics{Figure6.pdf}
\end{center}
Study the positions:\\

\begin{center}
\includegraphics{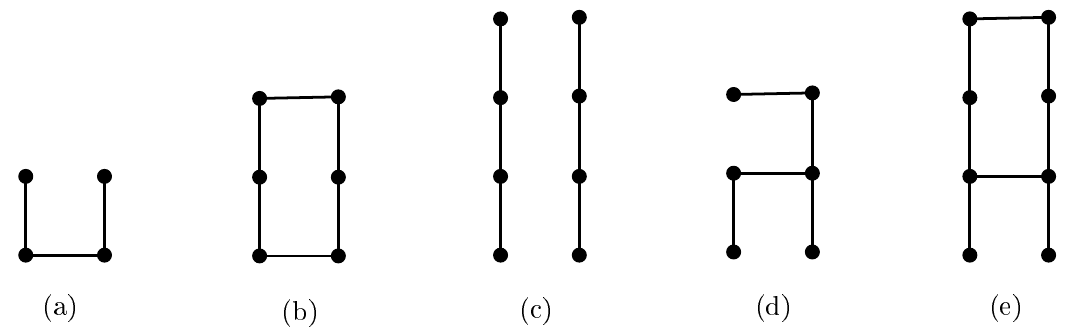}
\end{center}

All (a), (b), (c), (d), and (e) are $\mathcal{P}$-positions. The game value of (f) is $\moon=\moonform$.
\vspace{-1.2cm}
\begin{center}
\includegraphics{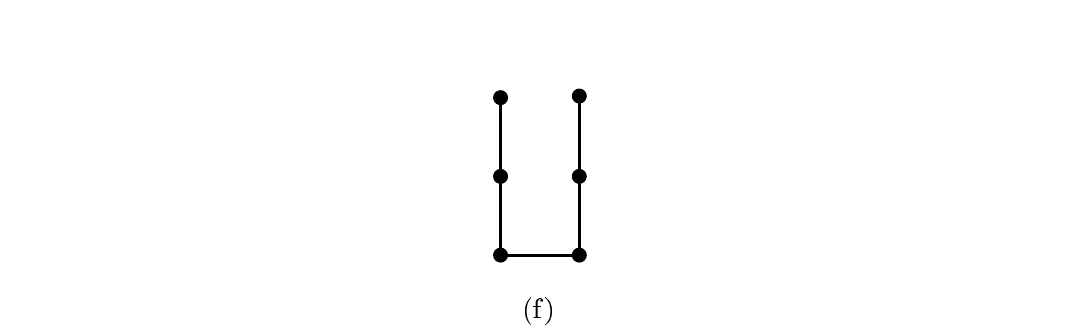}
\end{center}


Other positions that equal  \moon\ are the following.\\

\begin{center}
\includegraphics{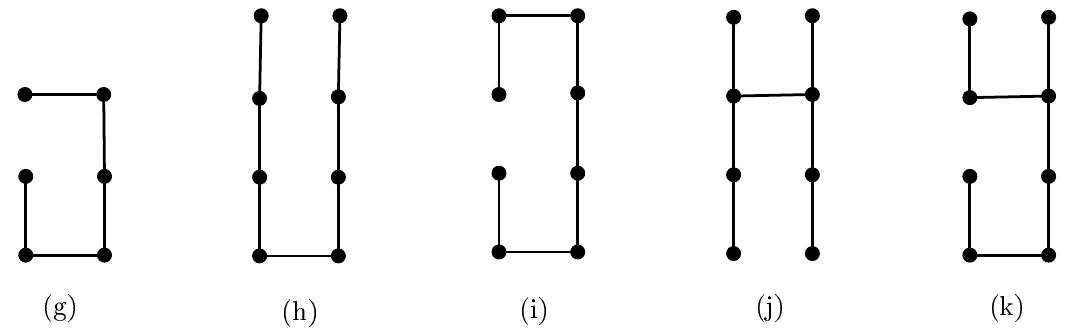}
\end{center}

In the position (l), the central horizontal move is the option (d), that is equal to $0$. The other options are (f) and (g), that are equal to \moon. Therefore, the literal form is $$l=\cg{0,\moon,\moon}{0,\moon,\moon}$$ with $S_l=\{0\}$, and $P_l=\varnothing$. Applying the affine impartial minimum excluded rule, we conclude that the position is $*$.

\vspace{-1.2cm}
\begin{center}
\includegraphics{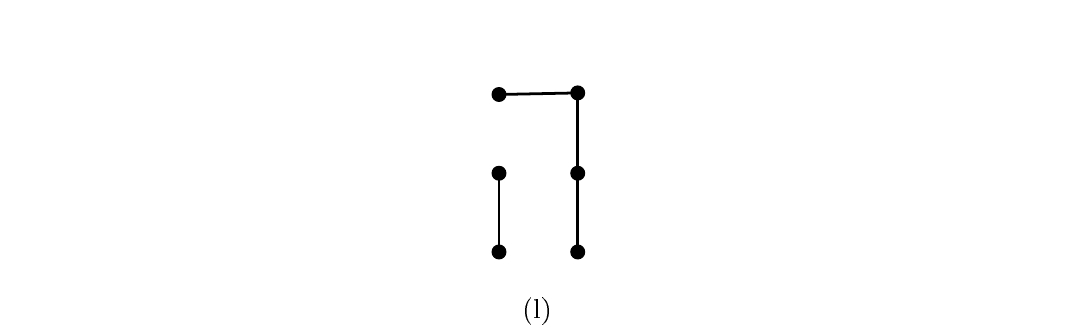}
\end{center}

The position (m) is also equal to $*$, i.e. $0+*$.

\begin{center}
\includegraphics{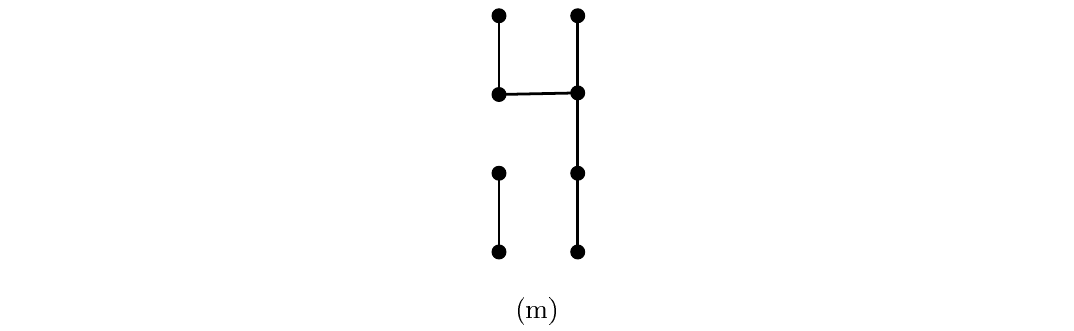}
\end{center}

Now, we are ready for (n), a more complex situation. The literal form is $$n=\{h,i,k,\{\infty\,|\,l\}\,|\,h,i,k,\{l\,|\,\overline{\infty}\}\},$$ that is, $$n=\{\leftmoon,\leftmoon,\leftmoon,\{\infty\,|\,*\}\,|\,\leftmoon,\leftmoon,\leftmoon,\{*\,|\,\overline{\infty}\}\}.$$ Hence, $S_n=\varnothing$, and $P_n=\Im\setminus\{*\}$. Applying the affine impartial minimum excluded rule, we conclude that the position is $*$.\\\\

\begin{center}
\includegraphics{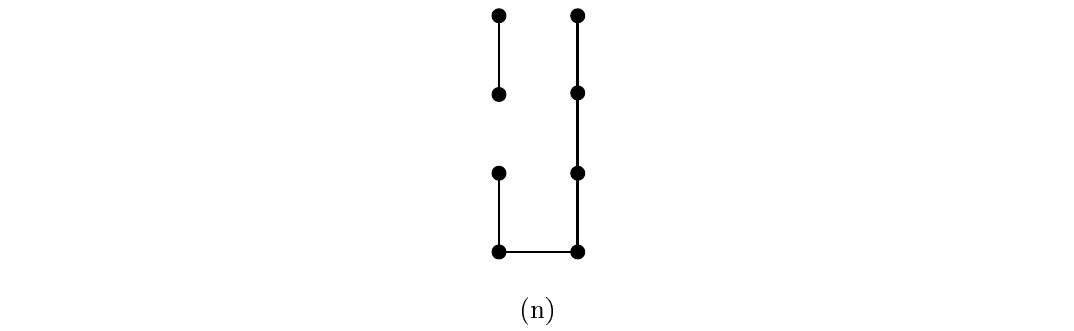}
\end{center}

Going back to the original question, we have the following.\\

\begin{center}
\includegraphics{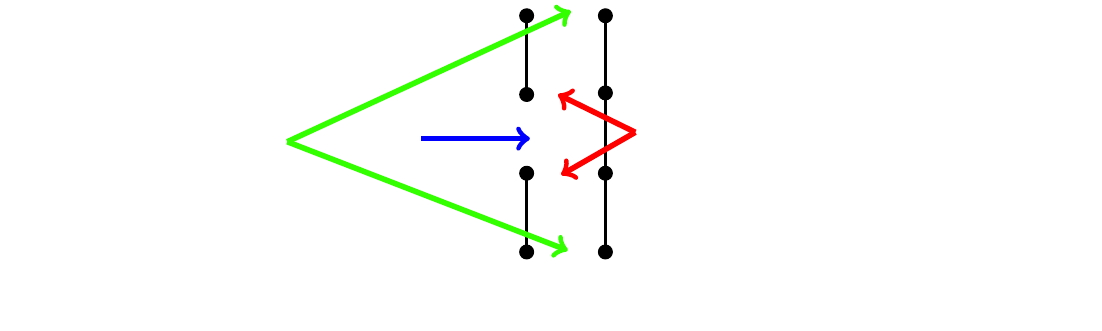}
\end{center}
The form is $\{\textbf{\textcolor[rgb]{0.00,1.00,0.00}{n}},\textbf{\textcolor[rgb]{1.00,0.00,0.00}{m}},\textbf{\textcolor[rgb]{0.00,0.07,1.00}{c}}\,|\,\textbf{\textcolor[rgb]{0.00,1.00,0.00}{n}},\textbf{\textcolor[rgb]{1.00,0.00,0.00}{m}},\textbf{\textcolor[rgb]{0.00,0.07,1.00}{c}}\}$, that is, $\{*,*,0\,|\,*,*,0\}=*2$. Here {\bf n} represents a play of the top or bottom bar, {\bf m} represents a play of some middle bar, and {\bf c} represents play of the left line. 


\section{Case study: {\sc top entails}}\label{sec:topentail}


We denote by $\boldsymbol{n}$ the literal form of a stack of size $n$. The literal form of the Left removal of the top coin from a stack of size $n$ is  $\{\infty\,|\,(\boldsymbol{n-1})^\mathcal{R}\}$ (and the symmetric from Right's point of view). With that in mind, let us compute the first few values.  First, we do it the tedious way, and then later after Theorem~\ref{th:slick}, we propose the slick recursive way for a few more values, in a table format.

Of course, $\boldsymbol{0}=\{\overline{\infty}\,|\,\infty\}$. The first player loses. Moreover,

\begin{itemize}
\item[]
$\boldsymbol{1}=\{\{\infty\,|\,\boldsymbol{0}^\mathcal{R}\}\,|\,\{\boldsymbol{0}^\mathcal{L}\,|\,\overline{\infty}\}\}=\{\{\infty\,|\,\infty\}\,|\,\{\overline{\infty}\,|\,\overline{\infty}\}\}$. Therefore, $S_{\boldsymbol{1}}=\varnothing$ and $P_{\boldsymbol{1}}=\Im$. Using the affine impartial minimum excluded rule, $\boldsymbol{1}=\moon$. In the next step, for ease, we will use the form $\moon=\pm\infty$. 

\item[]
$\boldsymbol{2}=\{\boldsymbol{1}+\boldsymbol{1}, \{\infty\,|\,\boldsymbol{1}^\mathcal{R}\}\,|\,\boldsymbol{1}+\boldsymbol{1}, \{\boldsymbol{1}^\mathcal{L}\,|\,\overline{\infty}\}\}=\{\moon,\{\infty\,|\,\overline{\infty}\}\,|\,\moon,\{\infty\,|\,\overline{\infty}\}\}$. Therefore, $S_{\boldsymbol{2}}=\varnothing$ and $P_{\boldsymbol{2}}=\varnothing$. Using the affine impartial minimum excluded rule, $\boldsymbol{2}=0$.

\item[]
 $\boldsymbol{3}=\{\boldsymbol{1}+\boldsymbol{2},\{\infty\,|\,\boldsymbol{2}^\mathcal{R}\}\,|\,\boldsymbol{1}+\boldsymbol{2},\{\boldsymbol{2}^\mathcal{L}\,|\,\overline{\infty}\}\}$. This game is equal to \linebreak $\{\moon,\{\infty\,|\,\moon,\{\infty\,|\,\overline{\infty}\}\}\,|\,\moon,\{\moon,\{\infty\,|\,\overline{\infty}\}\,|\,\overline{\infty}\}\}$. Therefore, $S_{\boldsymbol{3}}=\varnothing$ and $P_{\boldsymbol{3}}=\Im$. Using the affine impartial minimum excluded rule, $\boldsymbol{3}=\moon$. In the next step, for ease, we will use the form $\moon=\pm\infty$.

\item[]
 $\boldsymbol{4}=\{\boldsymbol{1}+\boldsymbol{3},\boldsymbol{2}+\boldsymbol{2},\{\infty\,|\,\boldsymbol{3}^\mathcal{R}\}\,|\,\boldsymbol{1}+\boldsymbol{3},\boldsymbol{2}+\boldsymbol{2},\{\boldsymbol{3}^\mathcal{L}\,|\,\overline{\infty}\}\}$. This game is equal to\linebreak$\{\moon,0,\{\infty\,|\,\overline{\infty}\}\,|\,\moon,0,\{\infty\,|\,\overline{\infty}\}\}$. Therefore, $S_{\boldsymbol{4}}=\{0\}$ and $P_{\boldsymbol{4}}=\varnothing$. Using the affine impartial minimum excluded rule, $\boldsymbol{4}=*$.

\item[]
$\boldsymbol{5}=\{\boldsymbol{1}+\boldsymbol{4},\boldsymbol{2}+\boldsymbol{3},\{\infty\,|\,\boldsymbol{4}^\mathcal{R}\}\,|\,\boldsymbol{1}+\boldsymbol{4},\boldsymbol{2}+\boldsymbol{3},\{\boldsymbol{4}^\mathcal{L}\,|\,\overline{\infty}\}\}$. This game is equal to $\{\moon,\moon,\{\infty\,|\,0\}\,|\,\moon,\moon,\{0\,|\,\overline{\infty}\}\}$. Therefore, $S_{\boldsymbol{5}}=\varnothing$ and $P_{\boldsymbol{5}}=\Im\setminus\{0\}$. Using the affine impartial minimum excluded rule, $\boldsymbol{5}=0$.

\item[]
$\boldsymbol{6}=\{\boldsymbol{1}+\boldsymbol{5},\boldsymbol{2}+\boldsymbol{4},\boldsymbol{3}+\boldsymbol{3},\{\infty\,|\,\boldsymbol{5}^\mathcal{R}\}\,|\,\boldsymbol{1}+\boldsymbol{5},\boldsymbol{2}+\boldsymbol{4},\boldsymbol{3}+\boldsymbol{3},\{\boldsymbol{5}^\mathcal{L}\,|\,\overline{\infty}\}\}$. This game is equal to $\{\moon,*,\moon,\{\infty\,|\,\{0\,|\,\overline{\infty}\}\}\,|\,\moon,*,\moon,\{\{\infty\,|\,0\}\,|\,\overline{\infty}\}\}$. Therefore, $S_{\boldsymbol{6}}=\{*\}$ and $P_{\boldsymbol{6}}=\{0\}$. Using the affine impartial minimum excluded rule, $\boldsymbol{6}=*2$.

\item[]
$\boldsymbol{7}=\{\boldsymbol{1}+\boldsymbol{6},\boldsymbol{2}+\boldsymbol{5},\boldsymbol{3}+\boldsymbol{4},\{\infty\,|\,\boldsymbol{6}^\mathcal{R}\}\,|\,\boldsymbol{1}+\boldsymbol{6},\boldsymbol{2}+\boldsymbol{5},\boldsymbol{3}+\boldsymbol{4},\{\boldsymbol{6}^\mathcal{L}\,|\,\overline{\infty}\}\}$. This game is equal to $\{\moon,0, \moon,\{\infty\,|\,*,\{\{\infty\,|\,0\}\,|\,\overline{\infty}\}\}\,|\,\moon,0,\moon,\{*,\{\infty\,|\,\{0\,|\,\overline{\infty}\}\}\,|\,\overline{\infty}\}\}$. So, $S_{\boldsymbol{7}}=\{0\}$, $P_{\boldsymbol{6}}=\Im\setminus\{0,*\}$, and with the affine impartial minimum excluded rule, $\boldsymbol{7}=*$.
\end{itemize}

Consider a stack of size $n$. We claim that an entailing move by Left does not protect her against an element in $S_{\boldsymbol{n-1}}$. To see this, let $*m\in S_{\boldsymbol{n-1}}$. Moving in $\boldsymbol{n}+*m$, if Left chooses $\{\infty\,|\,(\boldsymbol{n-1})\mathcal{^R}\}+*m$, Right answers $*m+*m$ and wins. On the other hand,  we observe that an entailing  move by Left does not protect her against the elements of $P_{\boldsymbol{n-1}}$. To see this, let $*m$ be an element of $P_{\boldsymbol{n-1}}$. Moving in $\boldsymbol{n}+*m$, if Left chooses $\{\infty\,|\,(\boldsymbol{n-1})\mathcal{^R}\}+*m$, because in $\boldsymbol{n-1}$, Right is protected against $*m$, he has an entailing winning move in the first component. Therefore, we have the general recursion
$$P_{\boldsymbol{n}}=\Im\setminus(S_{\boldsymbol{n-1}}\cup P_{\boldsymbol{n-1}}).$$
The set $S_{\boldsymbol n}$ is composed of the values of the positions of the form $\boldsymbol{\ell}+\boldsymbol m$, $\ell+m=n$, $\ell,m>0$, and disregarding any sum where \moon\ appears. Hence, the recurrence of {\sc top~entails} is as follows.

\begin{theorem}\label{th:slick}
The sets $P_0=S_0=\varnothing$, and for all $n>0$
$P_{\boldsymbol{n}}=\Im\setminus(S_{\boldsymbol{n-1}}\cup P_{\boldsymbol{n-1}}),$
$S_{\boldsymbol n}=\{\mathcal G(\boldsymbol{\ell}+\boldsymbol m), \boldsymbol{\ell},\boldsymbol m\ne\moon\}$.
\end{theorem}
\begin{proof}
This is explained in the above paragraph.
\end{proof}

Now, we can fill a table in an easy way.

\begin{center}
\begin{tabular}{|c|c|c|c|c|}
  \hline
  $n$ & $S_{\boldsymbol{n}}$ & $P_{\boldsymbol{n}}$ & $S_{\boldsymbol{n}}\cup P_{\boldsymbol{n}}$& $\mathcal G$-value (mex rule) \\
    \hline
 $0$ & $\varnothing$ & $\varnothing$ & $\varnothing$ & $0$ \\
  \hline
 $1$ & $\varnothing$ & $\Im$ & $\Im$ & $\infty$ \\
   \hline
  $2$ &  $\varnothing$ &  $\varnothing$ &  $\varnothing$ & $0$ \\
    \hline
      $3$ &  $\varnothing$ &  $\Im$ &  $\Im$ & $\infty$ \\
    \hline
          $4$ &  $\{0\}$ &  $\varnothing$ &  $\{0\}$ & $1$ \\
    \hline
              $5$ &  $\varnothing$ &  $\Im\backslash\{0\}$  &  $\Im\backslash\{0\}$ & $0$ \\
    \hline
                  $6$ &  $\{*\}$ &  $\{0\}$  &  $\{0,*\}$ & $2$ \\
    \hline
                      $7$ &  $\{0\}$ &  $\Im\backslash\{0,*\}$  &  $\Im\backslash\{*\}$ & $1$ \\
    \hline
                          $8$ &  $\{0,*2\}$ &  $\{*\}$  &  $\{0,*,*2\}$ & $3$ \\
    \hline
                              $9$ &  $\{*\}$ &  $\Im\backslash\{0,*,*2\}$  &  $\Im\backslash\{0,*2\}$ & $0$ \\
    \hline
                                  $10$ &  $\{0,*3\}$ &  $\{0,*2\}$  &  $\{0,*2,*3\}$ & $1$ \\
    \hline
                                      $11$ &  $\{0,*2\}$ &  $\Im\backslash\{0,*2,*3\}$  &  $\Im\backslash\{*3\}$ & $3$ \\
    \hline
                                          $12$ &  $\{0,*,*2\}$ &  $\{*3\}$  &  $\{0,*,*2,*3\}$ & $4$ \\
    \hline
\end{tabular}
\end{center}

With the recursion, we know that $\boldsymbol{n}=\moon$ if and only if $S_{\boldsymbol{n-1}}\cup P_{\boldsymbol{n-1}}\subseteq S_{\boldsymbol{n}}$. That happens for $n=2403$, $n=2505$, and $n=33243$, as mentioned in \cite{West1996}. One of three possibilities must happen: a) A finite number of finite nimbers; b) A finite number of loony values; c) An infinite number of finite nimbers and an infinite number of loony values. However, it is an open problem to know what case happens.\\

At the first Combinatorial Games Workshop at MSRI, John Conway proposed that an effort should be made to devise some game with entailing moves that is non-trivial, but (unlike {\sc top entails}) susceptible to a complete analysis. All attempts which have been tried turn out to be not very interesting. As a sequel to this work, we are finalizing a paper \cite{LNS2} with a proposal of a ruleset to meet Conway's suggestion.

\end{document}